\documentclass{amsart}

\usepackage{amssymb}
\usepackage{latexsym}
\usepackage{amsmath}
\usepackage{euscript}
\usepackage{graphics}

      \def\sF{{\mathfrak F}}

      \def\dC{{\mathbb C}}

   \def\dN{{\mathbb N}}   
      \def\dR{{\mathbb R}}

      \def\cL{{\mathcal L}}
      
\def\cP{{\mathcal P}}   \def\cQ{{\mathcal Q}}

\def\bm\chi{\mbox{\boldmath$\chi$}}

\def\diag{{\rm diag\,}}

\let\xker=\ker \def\ker{{\xker\,}}

\def\sg{\operatorname{sign}}

\def\sg{\operatorname{sign}}

\def\deg{\operatorname{deg}}

\newtheorem{theorem}{Theorem}[section]
\newtheorem{proposition}[theorem]{Proposition}

\newtheorem{definition}[theorem]{Definition}
\theoremstyle{remark}

\newtheorem{remark}[theorem]{Remark}

\numberwithin{equation}{section}

\author{Maxim Derevyagin}
\address{
Maxim Derevyagin\\
KU Leuven\\
Department of Mathematics\\
Celestijnenlaan 200B box 2400\\
BE-3001 Leuven}
\email{derevyagin.m@gmail.com}

\date{\today}

\dedicatory{\it Dedicated to the memory of  Herbert Stahl and Andrei Aleksandrovich Gonchar}

\begin{document}

\title[Spectral theory of $G$-symmetric tridiagonal matrices]{Spectral theory of 
the $G$-symmetric tridiagonal matrices related to Stahl's counterexample}

\begin{abstract} 
We recast Stahl's counterexample from the point of view of the spectral theory of the underlying non-symmetric Jacobi matrices.
In particular, it is shown that these matrices are self-adjoint and non-negative in a Krein space and have empty resolvent sets. In fact, 
the technique of Darboux transformations (aka commutation methods) on spectra which is used in the present paper
allows us to treat the class of all $G$-non-negative tridiagonal matrices. We also establish a correspondence
between this class of matrices and the class of signed measures with one sign change.
Finally, it is proved that the absence of the spurious pole at infinity for Pad\'e approximants is equivalent to the definitizability
of the corresponding tridiagonal matrix. 
\end{abstract}
\maketitle

\section{Introduction}

The famous Markov theorem in the theory of rational approximation states that diagonal  Pad\'e approximants for the Cauchy transform $F$ of a positive measure supported on $[-1,1]$ converge to $F$ locally uniformly in ${\mathbb C}\setminus[-1,1]$. In \cite{St85} Herbert Stahl gave his marvelous example of a very simply constructed non-positive weight function that 
particularly demonstrated impossibility to extend the Markov theorem to the case of signed measures.
The proposed weight function is a polynomial modification of the Chebyshev weight function:
\[
w(t)=\frac{(t-\cos\pi\alpha_1)(t-\cos\pi\alpha_2)}{\pi\sqrt{1-t^2}}, \quad t\in(-1,1),
\]
where $\alpha_1,\alpha_2\in(0,1)$ are rationally independent real numbers. This weight function possesses the property that
the entire sequence of the diagonal Pad\'e approximant to the Cauchy transform
\begin{equation}\label{StahlEx}
\sF(\lambda)=\int_{-1}^{1}\frac{1}{t-\lambda}w(t)dt
\end{equation}
exists and their poles are asymptotically dense everywhere in the complex plane $\dC$. Recall that $n$-th diagonal Pad\'e approximant
to the analytic function $\sF$ is a rational function $\displaystyle{\frac{\cQ_n}{\cP_n}}$ such that 
\[
\sF(\lambda)-\frac{\cQ_n(\lambda)}{\cP_n(\lambda)}=O\left(\frac{1}{\lambda^{2n+1}}\right), \quad \lambda\to\infty,
\]
$\deg\cP_n=n$ and $\deg\cQ_n\le n$.
If such a rational function exists then the polynomial $\cP_n$ turns out to be the polynomial of degree $n$ orthogonal with respect
to the signed measure $w(t)dt$
\[
\int_{-1}^{1}t^k\cP_n(t)w(t)dt=0,\quad k=0,\dots, n-1.
\]
Thus, the zeroes of $\cP_n$ from Stahl's counterexample are dense in $\dC$ although the signed measure is supported on the interval
$[-1,1]$. This is quite unexpected if one thinks in the streamline of the classical theory related to positive measures and that is why
one cannot prove the locally uniform convergence of diagonal Pad\'e approximants for the whole class of signed measures.

Nevertheless, there is still a hope to characterize classes of signed measures for which the Markov theorem can be extended 
like it is done for positive measures by means of self-adjointness of Jacobi matrices in a Hilbert space $\ell^2$.
For that purpose, it makes sense to understand the spectral picture behind Stahl's counterexample. Indeed, it is a standard fact 
that a sequence of orthogonal polynomials leads to a semi-infinite tridiagonal matrix. It is clear from the classical theory \cite{Ach61}, \cite{Si}
that, in the case of signed measures, the corresponding tridiagonal matrix cannot be symmetric in the usual sense and therefore
the underlying spectral theory becomes less transparent  but feasible. 

The main aim of the present note is to show the spectral mechanism of Stahl's counterexample and similar ones. However, to do so we will use
a slight simplification of the just considered example that was also pointed out by Herbert Stahl in \cite{St98}. In this simplification 
the weight function is simpler and is given by the formula
\[
w_0(t)=\frac{(t-\cos\pi\alpha)}{\pi\sqrt{1-t^2}}, \quad t\in(-1,1),
\]
where $\alpha\in(0,1)$ is irrational. Nevertheless, the zeroes of the polynomials orthogonal with respect to the signed measure
$w_0(t)dt$ are still dense in the real line $\dR$ and, hence, any point in $\dR\setminus[-1,1]$ is a spurious pole
(see also \cite{D2004} for a more general concept of spectral pollution). Basically, the nature of the signed measure $w_0(t)dt$ is the same as the nature of $w(t)dt$. Moreover, using the technique of Darboux transformations one can already see spectral effects appearing in the case of $w_0(t)dt$ and similar signed measures. 

Before going into details of non-symmetric matrices and signed measures, let us briefly recall some basic facts about
symmetric tridiagonal matrices and positive measures.
So, let $d\mu$ be a positive probability measure and let the support of this measure be equal to or contained in the interval $[-1,1]$. In this case, one can construct a sequence
of polynomials $P_n$ orthogonal with respect to $d\mu$, i.e.
\begin{equation*}
\int_{-1}^{1}P_n(t)P_m(t)d\mu(t)=\delta_{nm}, \quad n,m=0,1,\dots,
\end{equation*}
where $\delta_{nm}$ is the Kronecker delta. It is well known that the polynomials $P_n$ satisfy the second order difference equation:
\begin{equation}\label{TriTerm}
a_kP_{k+1}(t)+b_kP_k(t)+a_{k-1}P_{k-1}(t)=xP_k(x), \quad k=0,1,\dots
\end{equation}
with the initial conditions
\begin{equation}\label{InCon}
P_{-1}=0,\quad P_0=0;
\end{equation}
here $b_k\in\dR$ and $a_k>0$ for $k=0,1,\dots$. 
Sometimes it is also convenient to have \eqref{TriTerm}, \eqref{InCon} in the following matrix form
\begin{equation}\label{MatrixForm}
Jp(t)=tp(t),
\end{equation}
where $p=(P_0, P_1, P_2,\dots)^{\top}$ and 
\begin{equation}\label{JacOp}
J=\begin{pmatrix}
b_0 & a_0 & 0 & \cdots   \\
     a_0 & b_1 & a_1 &   \\
     0 & a_1 & b_2 & \ddots  \\
     \vdots & & \ddots & \ddots
\end{pmatrix}
\end{equation}
is a Jacobi matrix, which is a symmetric tridiagonal matrix. We will say that $J$ corresponds to the measure $d\mu$.
In the standard way, with such a Jacobi matrix one can associate an operator acting in the Hilbert space $\ell^2$
of square summable sequences of complex numbers. This operator will be also denoted by $J$. Since $d\mu$
is compactly supported, the Jacobi operator is bounded. Moreover, the spectrum $\sigma(J)$ of $J$ coincides with the support
of the measure $d\mu$ \cite{Ach61}, \cite{Si}. 

The paper is organized as follows. In Section 2, the shifted Darboux transformations $\widetilde{J}(x)$ of $J$ for $x\in\sigma(J)$ are presented.
Moreover, a class of $G$-symmetric tridiagonal matrices associated to signed measures with one sign change is characterized
by means of the shifted Darboux transformations. The next section deals with the $G$-symmetric matrix $\widetilde{J}_0(\cos\pi\alpha)$
corresponding to Stahl's weight $w_0(t)dt$. Particularly, $G$-self-adjointness of the operator $\widetilde{J}_0(\cos\pi\alpha)$ is shown. 
The main result of Section 4 is
the fact that $\widetilde{J}_0(\cos\pi\alpha)$ has an empty resolvent set $\rho(\widetilde{J}_0(\cos\pi\alpha))$. Finally,
in Section 5 it is proved that the definitizability of $\widetilde{J}(x)$ is equivalent to the absence of the spurious pole at infinity for
the underlying diagonal Pad\'e approximants.

\section{Darboux transformations and $G$-symmetric matrices}

In this section we recall how $LU$-factorizations and Darboux transformations are related and formulate it in an appropriate manner.
This will allow us to present here a Favard-type theorem for signed measures with one sign change. 

Unlike the standard monic approach \cite{BM04}, in the present paper we consider the case of symmetric tridiagonal matrices 
associated with positive measures.  
To get to Darboux transformations, let us start by factorizing the tridiagonal matrix $J-xI$ as follows
\begin{equation}\label{LU}
J-xI = \cL(x) D(x) \cL^{\top}(x),
\end{equation}
where $D(x)=\diag(d_0(x),d_1(x),....)$ is a diagonal matrix and $\cL$ is a lower bidiagonal matrix 
\[      
\cL =\begin{pmatrix}
1 & 0 & 0 & \cdots   \\
     v_0 & 1 & 0 &   \\
     0 & v_1 & 1 & \ddots  \\
     \vdots & & \ddots & \ddots
\end{pmatrix}.
\]
Comparing coefficients in \eqref{LU} gives
\begin{equation}\label{LU_entries1}
d_0(x)=b_0-x, \quad    
      d_j(x)v_j(x) = a_{j} , \quad
      d_{j+1}(x)=b_{j+1}-x-d_j(x)v_j^2(x).
\end{equation}
Now it is easy to check with the help of \eqref{TriTerm} that  
\begin{equation}\label{LU_entries}
d_j(x)=-a_j\frac{P_{j+1}(x)}{P_j(x)},\quad 
v_j(x) =-
\frac{P_{j}(x)}{P_{j+1}(x)}.
\end{equation}
It will be more convenient to rewrite the factorization \eqref{LU} in the following way
\begin{equation}\label{LU_C}
J-xI = L(x) G(x) L^{\top}(x),
\end{equation}
where $G(x)=\sg D(x)=\diag(\sg(d_0(x)), \sg(d_1(x)), \sg(d_2(x)),\dots)$ and
\begin{equation}\label{DefC}
L=\cL|D|^{1/2}=\begin{pmatrix}
\sqrt{|d_0|} & 0 & 0 & \cdots   \\
     v_0\sqrt{|d_0|} & \sqrt{|d_1|} & 0 &   \\
     0 & v_1\sqrt{|d_1|} & \sqrt{|d_2|} & \ddots  \\
     \vdots & & \ddots & \ddots
\end{pmatrix}.
\end{equation}

In order to define the main object of the present study we will need the following statement, which is a specification
of the known results that appear in the context of numerical algorithms, orthogonal polynomials, spectral theory, and integrable systems 
\cite{AM}, \cite{BM04}, \cite{Gau02}, \cite{GT96}, \cite{Zh97}.
\begin{proposition}\label{DarbouxTr}
Let J be a Jacobi matrix corresponding to $d\mu$. Assume that for $J$ and $x\in(-1,1)$ the factorization
of the form
\begin{equation}\label{DarFac}
J-xI=LGL^{\top}
\end{equation}
exists. Then the tridiagonal matrix 
\[
\widetilde{J}(x)=GL^{\top}L+xI
\]
corresponds to the signed measure $(t-x)d\mu(t)$ supported on $[-1,1]$.
\end{proposition}

\begin{remark}
It should be stressed here that the existence of the factorization \eqref{DarFac} is not essential. 
Indeed, a similar factorization, which is called the Bunch-Kaufman factorization, holds for any Jacobi matrix but one has to
replace the diagonal matrix $G$ with a $2\times 2$ block diagonal matrix (for instance, see \cite{BK77}, \cite{D13}, \cite{DD10}).
To avoid a mess with indices and some technicalities we always assume that the factorization \eqref{DarFac}  exists but based
on \cite{DD10} one can easily adapt all the results of this paper to the general case.
\end{remark}

\begin{proof}
The detailed proof can be found in \cite{BM04}. We will only show how to see that $\widetilde{J}(x)$ 
corresponds to $(t-x)d\mu(t)$. To this end let us write the matrix form \eqref{MatrixForm} of the three-term recurrence relation 
\[
Jp(t)=tp(t).
\]
Now one can see that the chain of implications 
\[
\begin{split}
(J-xI)p(t)=(t-x)p(t)\Rightarrow LGL^{\top}p(t)=(t-x)p(t)\Rightarrow\\
\Rightarrow (GL^{\top}L)GL^{\top}p(t)=(t-x)GL^{\top}p(t)\Rightarrow 
\widetilde{J}(x)GL^{\top}p(t)=tGL^{\top}p(t)
\end{split}
\]
suggests that $GL^{\top}p(t)$ is a vector of the orthogonal polynomials corresponding to $\widetilde{J}(x)$.
However, the first entry of $GL^{\top}p(t)$ vanishes at $t=x$ and thus $GL^{\top}p(t)$ doesn't satisfy the proper initial 
condition $\widetilde{P}_0=1$.
Nevertheless, introducing
\[
\widetilde{p}(t)=\frac{1}{t-x}GL^{\top}p(t)=(\widetilde{P}_0(t), \widetilde{P}_1(t), \dots)^{\top}
\]
solves the problem. Moreover, it is not so difficult to check that the polynomials
\[
\widetilde{P}_j(t)=\sg(d_j(x))\sqrt{|d_j(x)|}\frac{P_j(t)-\frac{P_j(x)}{P_{j+1}(x)}P_{j+1}(t)}{t-x}
\]
are orthogonal with respect to $(t-x)d\mu(t)$ (this can also be thought as a special case of the Christoffel formula).
\end{proof}

The matrix $\widetilde{J}(x)$ is called the shifted Darboux transformation of $J$. 
Since this transformation is based on the Christoffel formula, it is sometimes called the Christoffel transformation \cite{BM04}, \cite{Zh97}.
As a matter of fact, the shifted Darboux transformation of a symmetric matrix is not symmetric but it does possess symmetry properties.
To see this recall that a matrix $H$ is called $G$-symmetric if the matrix $GH$ is symmetric. Obviously, $\widetilde{J}(x)$ is a $G$-symmetric
for the corresponding diagonal matrix $G$. It is worth mentioning that spectral properties of a $G$-symmetric matrix $H$ are equivalent to the corresponding ones of the linear pencil $GH-\lambda G$ (for example, see \cite{DL} where a similar class of matrices was studied).

The following result can be considered as a Favard-type theorem for signed measures with one sign change.

\begin{theorem}
Let $G$ be a diagonal matrix with +1 and -1 on the diagonal. A $G$-symmetric tridiagonal matrix $\widetilde{J}$ corresponds to
the signed measure 
\begin{equation}\label{SM_f}
(t-x)d\mu(t),
\end{equation}
where $d\mu$ is a positive measure, if and only if $\widetilde{J}$ admits the representation
\begin{equation}\label{Gnonn}
\widetilde{J}=GL^{\top}L+xI,
\end{equation}
where $L$ is a bidiagonal matrix of the form \eqref{DefC} and $x$ is a real number.
\end{theorem}
\begin{remark}
It should be emphasized that this statement is rather algebraic and it neither requires the support of the measure to be bounded nor provides with
the condition for the support of the measure to be bounded.
\end{remark}
\begin{proof}
The "only if" part is Proposition \ref{DarbouxTr}. To prove the "if" part notice that if we have matrices $L$ and $G$ then
we can construct the symmetric Jacobi matrix as follows
\[
J=LGL^{\top}+xI.
\]
This Jacobi matrix corresponds to a positive measure $d\mu$. The rest follows from the uniqueness of the Cholesky
factorization \eqref{LU_C} and Proposition \ref{DarbouxTr}.
\end{proof}

\begin{remark} Recall that a matrix $H$ is called $G$-non-negative if the matrix $GH$ is non-negative. It is clear that
the tridiagonal matrix $\widetilde{J}$ can be represented in the form \eqref{Gnonn} if and only if $\widetilde{J}-xI$ is
$G$-non-negative for some diagonal matrix $G$ and $x\in\dR$.
\end{remark}
This Favard-type theorem gives an efficient way to construct tridiagonal matrices corresponding to signed measures of the form \eqref{SM_f}.
Indeed, it would be natural to start with a sequence of signs $\pm 1$ and two bounded sequences of positive numbers. These data produce matrices $G$ and $L$ and, so, a $G$-symmetric matrix $\widetilde{J}$. However, as we will see later, it is impossible to get all the signed measures 
of the form \eqref{SM_f} with bounded support using only bounded sequences. 

\section{Tridiagonal matrices associated with Stahl's counterexample}

Here we translate Stahl's counterexample to the language of tridiagonal matrices. 

Let us first recall that  orthonormal Chebyshev polynomials are defined as 
\begin{equation}\label{ChebyshevP}
T_0(t)=1,\quad T_n(t)=\sqrt{2}\cos(n\arccos t),\quad t\in [-1,1], \quad n=1,2,\dots.
\end{equation}  
It is not so difficult to verify that they satisfy the following difference equation:
\[
\begin{split}
tT_0(t)&=\frac{1}{\sqrt{2}}T_1(t),\\
tT_1(t)&=\frac{1}{\sqrt{2}}T_0(t)+\frac{1}{2}T_2(t),\\
tT_k(t)&=\frac{1}{2}T_{k-1}(t)+\frac{1}{2}T_{k+1}(t), \quad k=2,3,\dots
\end{split}
\]
and the orthogonality relation
\[
\int_{-1}^{1}T_n(t)T_m(t)\frac{dt}{\pi\sqrt{1-t^2}}=\delta_{n m}.
\]
Thus, the Jacobi matrix corresponding to $d\mu_0(t)=\frac{dt}{\pi\sqrt{1-t^2}}$ has the following form
\begin{equation}\label{JacForCheb}
J_0=
\begin{pmatrix}
0 & \frac{1}{\sqrt{2}} & 0 & \cdots   \\
     \frac{1}{\sqrt{2}} & 0 & \frac{1}{{2}} &   \\
     0 & \frac{1}{{2}} & 0 & \ddots  \\
     \vdots & & \ddots & \ddots
\end{pmatrix}
\end{equation}  
and it is a standard fact that the spectrum $\sigma(J_0)$ of $J_0$ is $[-1,1]$ and, consequently, the resolvent set
$\rho(J_0)$ is $\dC\setminus[-1,1]$.

Taking into account the explicit form \eqref{ChebyshevP} of the Chebyshev polynomials,
one can define the shifted Darboux transformation of $J_0$ and the following phenomenon takes place. 
\begin{proposition}
Assume that $\alpha\in(0,1)$ is not rational. Then the shifted Darboux transformation $\widetilde{J_0}(\cos\pi\alpha)$ of the Jacobi matrix $J_0$ corresponds to
the signed measure 
\[
\frac{(t-\cos\pi\alpha)}{\pi\sqrt{1-t^2}}dt
\] 
supported on $[-1,1]$ and it generates an unbounded operator in $\ell^2$.
\end{proposition}
\begin{proof}
In the case of the Jacobi matrix corresponding to the Chebyshev polynomials, the entries of $L$ and $D$ in the factorization
\[
J_0-\cos\pi\alpha I=LDL^{\top}
\]
 are given by
\[
\begin{gathered}
d_0=-\frac{1}{\sqrt{2}}\frac{T_1(\cos\pi\alpha)}{T_0(\cos\pi\alpha)}=-\frac{1}{\sqrt{2}}\cos\pi\alpha, \quad
d_j=-\frac{1}{{2}}\frac{T_{j+1}(\cos\pi\alpha)}{T_j(\cos\pi\alpha)}=-\frac{1}{{2}}\frac{\cos(j+1)\pi\alpha}{\cos j\pi\alpha}\\
v_j=-\frac{T_j(\cos\pi\alpha)}{T_{j+1}(\cos\pi\alpha)}=-\frac{\cos j\pi\alpha}{\cos(j+1)\pi\alpha},
\end{gathered}
\]
where we have used \eqref{ChebyshevP} to simplify the expressions. We see that $T_j(\cos\pi\alpha)=\cos j\pi\alpha\ne 0$ because
of the irrationality of $\alpha$ and thus the factorization \eqref{LU} exists although 
$\cos\pi\alpha\in[-1,1]=\sigma(J_0)=\operatorname{supp}d\mu_0$. 
Next, straightforward calculations give
\[
\begin{split}
\widetilde{J_0}(\cos\pi\alpha)&=GL^{\top}L+\cos\pi\alpha I=\sg D |D|^{1/2}\cL\cL^{\top}|D|^{1/2}+\cos\pi\alpha I=\\
&=\begin{pmatrix}
d_0+v_0^2d_0 +\cos\pi\alpha & \frac{1}{\sqrt[4]{2}}\sqrt{|v_0d_1|} & 0 & \cdots   \\
  \frac{\varepsilon_0\varepsilon_1}{\sqrt[4]{2}}\sqrt{|v_0d_1|}  & d_1+v_1^2d_1+\cos\pi\alpha 
  & \frac{1}{\sqrt{2}}\sqrt{|v_1d_2|}  &   \\
     0 & \frac{\varepsilon_1\varepsilon_2}{\sqrt{2}}\sqrt{|v_1d_2|}& d_2+v_2^2d_2 +\cos\pi\alpha & \ddots  \\
     \vdots & & \ddots & \ddots
\end{pmatrix},
\end{split}
\]
where $\varepsilon_j=\sg d_j=\sg v_j$.
Now it is easy to see that $\widetilde{J_0}(\cos\pi\alpha)$ is unbounded in $\ell^2$. Indeed, for the standard unit vector $e_j\in\ell^2$ we have that
\[
\begin{split}
\Vert \widetilde{J_0}(\cos\pi\alpha)&\Vert\ge \Vert\widetilde{J_0}(\cos\pi\alpha)e_j\Vert=
\Vert\widetilde{J_0}(\cos\pi\alpha)e_j\Vert\Vert e_j\Vert\\
&\ge\left| (\widetilde{J_0}(\cos\pi\alpha)e_j,e_j)_{\ell^2}\right|=|d_j(1+v_j^2)+\cos\pi\alpha|\ge |d_j|-1.
\end{split}
\]
It remains to observe that the sequence
\[
d_j=-\frac{1}{{2}}\frac{\cos(j+1)\pi\alpha}{\cos j\pi\alpha}=\frac{1}{2}(\cos\pi\alpha-\sin\pi\alpha\tan j\pi\alpha)
\]
is not bounded since the set $\{\tan j\pi\alpha\}_{j=0}^{\infty}$ is dense in $\dR$ due to Kronecker's density theorem.
\end{proof}

To  demonstrate more properties of $\widetilde{J_0}(\cos\pi\alpha)$ we need to give some basic definitions related to
operators in spaces with indefinite inner products.
Let us begin by noticing that the diagonal matrix $G$ given by \eqref{LU_C} possesses the following properties
\[
G^2=I,\quad G^*=G,
\]
where $G^*$ is the Hermitian  adjoint with respect to the Hilbert space $\ell^2$. Thus the operator $G$ induces an additional inner product
\[
[f,g]=(Gf,g)_{\ell^2}\quad f,g\in\ell^2.
\]
Therefore, the space $\ell^2(G)$, which is a linear space of elements of $\ell^2$ equipped with the bilinear form $[\cdot,\cdot]$, is 
a Krein space \cite{AI}.
By definition, the norm of a Krein space
is the same as the norm of the Hilbert space generating the Krein space~\cite{AI}. In particular, the norm of the Krein space $\ell^2(G)$ is 
the norm of the Hilbert space $\ell^2$.
\begin{theorem}\label{SA_th}
The operator $\widetilde{J_0}(\cos\pi\alpha)$ is self-adjoint in $\ell^2(G)$. Moreover, the operator 
$\widetilde{J_0}(\cos\pi\alpha)-\cos\pi\alpha I$ is non-negative
in $\ell^2(G)$, that is
\begin{equation}\label{SemiB}
\left[\left(\widetilde{J_0}(\cos\pi\alpha)-\cos\pi\alpha I\right)f,f\right]\ge 0
\end{equation}
for any $f$ from the domain of $\widetilde{J_0}(\cos\pi\alpha)$.
\end{theorem}
\begin{proof}
It is clear from the definition of the Krein space that the self-adjointness of $\widetilde{J_0}(\cos\pi\alpha)$ in $\ell^2(G)$ is equivalent
to the self-adjointness of $G\widetilde{J_0}(\cos\pi\alpha)$ in $\ell^2$. So, let us show that the classical symmetric Jacobi matrix
\[
\begin{gathered}
G\widetilde{J_0}(\cos\pi\alpha)=L^{\top}L+\cos\pi\alpha G=\\
\begin{pmatrix}
\varepsilon_0(d_0+v_0^2d_0 +\cos\pi\alpha) & \frac{\varepsilon_0}{\sqrt[4]{2}}\sqrt{|v_0d_1|} & 0 & \cdots   \\
  \frac{\varepsilon_0}{\sqrt[4]{2}}\sqrt{|v_0d_1|}  & \varepsilon_1(d_1+v_1^2d_1+\cos\pi\alpha) 
  & \frac{\varepsilon_1}{\sqrt{2}}\sqrt{|v_1d_2|}  &   \\
     0 & \frac{\varepsilon_1}{\sqrt{2}}\sqrt{|v_1d_2|}& \varepsilon_2(d_2+v_2^2d_2 +\cos\pi\alpha) & \ddots  \\
     \vdots & & \ddots & \ddots
\end{pmatrix}
\end{gathered}
\]
is self-adjoint in $\ell^2$. To do this we will need Carleman's sufficiency condition \cite{Ach61} (see also \cite[Corollary 4.5]{Si}).
This sufficiency condition states that if the series 
\[
\sum_{k=0}^{\infty}\frac{1}{a_k}
\]
is divergent then the Jacobi operator \eqref{JacOp} is self-adjoint. In the case of $G\widetilde{J_0}(\cos\pi\alpha)$
we have the series 
\begin{equation}\label{SerCh}
\begin{split}
\sum_{k=1}^{\infty}\frac{1}{\frac{1}{\sqrt{2}}\sqrt{v_kd_{k+1}}}&=
{2}\sum_{k=1}^{\infty}
\sqrt{\frac{\cos^2(k+1)\pi\alpha}{|\cos k\pi\alpha\cdot\cos(k+2)\pi\alpha|}}
\\
&=2\sqrt{2}\sum_{k=1}^{\infty}
\sqrt{\frac{\cos^2(k+1)\pi\alpha}{|\cos 2\pi\alpha-1+2\cos^2(k+1)\pi\alpha|}}
\end{split}
\end{equation}
to check. According to Kronecker's density theorem the set $\{\cos(k+1)\pi\alpha\}_{k=1}^{\infty}$ is dense in $[-1,1]$. So, the series \eqref{SerCh} is divergent 
which proves the self-adjointness of $\widetilde{J_0}(\cos\pi\alpha)$ in $\ell^2(G)$. To see the non-negativity it is
enough to verify \eqref{SemiB} for finitely supported vectors $f$ which can be easily done by the definition of 
$\widetilde{J_0}(\cos\pi\alpha)$:
\begin{equation}\label{NonDar}
\begin{split}
\left[\left(\widetilde{J_0}(\cos\pi\alpha I)-\cos\pi\alpha I\right)f,f\right]&=[(GL^{\top}L)f,f]=(L^{\top}Lf,f)_{\ell^2}\\
&=(Lf,Lf)_{\ell^2}\ge 0.
\end{split}
\end{equation}
\end{proof}
Note that the definition of the shifted Darboux transformation was only used in \eqref{NonDar} and so it is true
for the shifted Darboux transformations of any symmetric Jacobi matrix $J$.

\section{Spectra of the Stahl tridiagonal operators}

In the previous section we showed the self-adjointness of $\widetilde{J_0}(\cos\pi\alpha)$ in $\ell^2(G)$.
However, knowing that an operator is self-adjoint in a Krein space doesn't say much because the spectrum of a self-adjoint operator
in a Krein space can be fairly arbitrary and the structure of the operator can be rather wild \cite{L1982}.
Thus, to expect reasonable properties from self-adjoint operators it makes sense to study narrower classes of operators in Krein spaces. For instance, one of such classes is the class of definitizable operators for which spectral calculus was constructed in \cite{J1981},  \cite{L1982}.

\begin{definition}[\cite{L1982}] A self-adjoint operator $A$ in the Krein space $\ell^2(G)$ is called
definitizable if the resolvent set $\rho(A)$ of $A$ is not empty and there exists a polynomial $h$ such that
\[
 [h(A)f,f]\ge 0, \quad f\in\ell^2(G),
\]
and $f$ has only a finite number of nonzero elements. 
\end{definition}
 
 Now one can see that due to Theorem \ref{SA_th} the operator $\widetilde{J_0}(\cos\pi\alpha)$ is a good candidate to be
 a definitizable operator in $\ell^2(G)$ with the definitizing polynomial $h(t)=t-\cos\pi\alpha$.
To proceed with that we need the following properties of definitizable operators.

\begin{proposition}[Corollaries 1 and 2, Section II.2 in \cite{L1982}]\label{SpecDef}
Let $A$ be a definitizable operator in $\ell^2(G)$. Then the following statements hold true.
\begin{enumerate}
\item[(i)] The spectrum of $A$ is not empty.
\item[(ii)] If $A$ has a bounded spectrum then $A$ is bounded.
\end{enumerate}
\end{proposition}

In order to verify whether $\widetilde{J_0}(\cos\pi\alpha)$ is definitizable in $\ell^2(G)$ it remains to understand what actually happens with the spectrum under the Darboux transformation.
To this end, recall that for bounded operators $A$ and $B$ in a Hilbert space the spectra of the products
$AB$ and $BA$ coincide away from zero. In the more general situation of unbounded operators $A$ and $B$
the following statement holds true.
\begin{proposition}[\cite{HKM00}]\label{HKM_th}
Suppose that $\rho(AB)\ne\emptyset$ and $\rho(BA)\ne\emptyset$. Then the relation
\[
\sigma(AB)\setminus\{0\}=\sigma(BA)\setminus\{0\}
\]
is valid.
\end{proposition}
Based on this theorem we can get the full information about the spectrum of the operator $\widetilde{J_0}(\cos\pi\alpha)$.
\begin{theorem}
We have that
\[
\sigma(\widetilde{J_0}(\cos\pi\alpha))=\dC
\]
and, thus, the $G$-self-adjoint operator $\widetilde{J_0}(\cos\pi\alpha)$ is not definitizable.
\end{theorem}
\begin{proof}
To begin with, observe that 
\[
[-1,1]-\cos\pi\alpha=[-1-\cos\pi\alpha,1-\cos\pi\alpha]=\sigma(J_0-\cos\pi\alpha I)=\sigma(LGL^{\top}).
\]
Now let us set $A=L$ and $B=GL^{\top}$. Then
\[
\sigma(BA)=\sigma(GL^{\top}L)=\sigma(\widetilde{J_0}(\cos\pi\alpha)-\cos\pi\alpha I)=
\sigma(\widetilde{J_0}(\cos\pi\alpha))-\cos\pi\alpha.
\]
Clearly, the resolvent set $\rho(AB)$ is not empty because $AB$ is bounded. If the relation
$\rho(BA)\ne\emptyset$ were true then Theorem \ref{SA_th} would imply the definitizability of 
$\widetilde{J_0}(\cos\pi\alpha)$ in $\ell^2(G)$ with the definitizing polynomial $h(t)=t-\cos\pi\alpha$ and, in turn,
from Proposition \ref{HKM_th} we would get that
$\sigma(\widetilde{J_0}(\cos\pi\alpha))=[-1,1]$ which cannot be true for an unbounded definitizable operator
because of Proposition \ref{SpecDef}.
\end{proof}

Noteworthy that another class of $G$-self-adjoint operators with empty spectra appeared in a different context and related
to Clifford algebras \cite{KT2011}, \cite{KP2012}. 
\begin{remark}\label{SimonEx}
As a matter of fact, a similar situation when the spectrum blows up to the whole complex plane after the Darboux transformations
happens even in the definite setting of positive measures. Really, it is shown in \cite[Corollary 4.21]{Si} that there is a self-adjoint
Jacobi operator $J=LL^{\top}$ (i.e. it corresponds to a determinate moment problem) such that the Darboux transformation
$\widetilde{J}=L^{\top}L$ is symmetric but not self-adjoint (i.e. it corresponds to an indeterminate moment problem).
In other words, $\sigma(J)\subseteq\dR$ and at the same time $\sigma(\widetilde{J})=\dC$.
\end{remark}

\section{Definitizability and spurious poles}

In this section a relation between definitizability and spurious poles of  Pad\'e approximants is shown. Recall that spurious poles
 are those poles of  Pad\'e approximants that do not correspond to analytic properties of the original function \cite{St98}.
 For instance, for the diagonal Pad\'e approximants to the function \eqref{StahlEx} any point in $\dC\setminus[-1,1]$
 is a spurious pole. It should also be pointed out that convergence of diagonal Pad\'e approximants appears as a strong resolvent
 convergence of the truncations of the underlying tridiagonal matrix \cite{DD07}, \cite{Si} and, so, the concept of spurious poles is a particular case of the effect
 of spectral pollution \cite{D2004}.

The class of analytic functions we are concerned with here is the set of the Cauchy transforms of signed measures of the form
\begin{equation}
\sF(\lambda)=\int_{-1}^{1}\frac{1}{t-\lambda}(t-x)d\mu(t),
\end{equation} 
where $d\mu$ is a positive probability measure and $x\in(-1,1)$. It is easy to see that $\sF$ can be represented as follows
\[
\sF(\lambda)=(\lambda-x)\int_{-1}^{1}\frac{1}{t-\lambda}d\mu(t)+1=(\lambda-x)F(\lambda)+1,
\]
where $F$ is a Markov function. Basically, any real rational perturbation $r_1F+r_2$ of a Markov function is called a definitizable
function \cite{J2000}. However, there are special classes of definitizable functions with rather different nature (for instance,
see \cite{DD07}, \cite{DD09}, \cite{Gon75}, \cite{Lop80}). In particular, the Cauchy transform $\sF$ belongs to the class of definitizable functions
studied in \cite{DD09}.

The main goal of this section is to give some spectral sense to the following result.

\begin{proposition}[Theorem 5.5 from \cite{DD09}]\label{DDth} The poles of the diagonal Pad\'e approximants to $\sF$ are contained in
a bounded subset of the real line if and only if
\begin{equation}\label{Bcond}
\sup_{j\in\dN}\left|a_j\frac{P_{j+1}(x)}{P_j(x)}\right|<\infty,
\end{equation} 
where $P_j$ is the orthogonal polynomial of degree $j$ corresponding to the positive measure $d\mu$ and $a_j$ is one of the recurrence coefficients \eqref{TriTerm}.
\end{proposition}

From  \eqref{LU_entries1} and \eqref{LU_entries} one can easily see that the boundedness \eqref{Bcond} is equivalent to
the boundedness in $\ell^2$ of the operator $L$ defined by \eqref{DefC} (see also \cite[Theorem 6.3]{DD10} where this was shown in the context of monic orthogonal polynomials).  Finally, having in mind the Stahl tridiagonal matrices one can conclude that the boundedness of $L$
is the key to the definitizability of the shifted Darboux transform $\widetilde{J}(x)$ in the corresponding Krein space $\ell^2(G)$.
All these statements can be summarized into the following form.

\begin{theorem}\label{Berlin}
 Let $J$ be a Jacobi matrix corresponding to the measure $d\mu$, whose support is equal to or contained
in the interval $[-1,1]$, and let $x\in(-1,1)$. Then the following statements are equivalent:
\begin{enumerate}
\item[(i)] The shifted Darboux transformation $\widetilde{J}(x)$ is a definitizable operator in $\ell^2(G)$.
\item[(ii)] The relation 
\begin{equation}\label{SpecD}
\sigma(\widetilde{J}(x))\setminus\{x\}=\sigma(J)\setminus\{x\}
\end{equation} 
is valid.
\item[(ii)] The set of the poles of diagonal Pad\'e approximants to the definitizable function
\[
\int_{-1}^{1}\frac{1}{t-\lambda}(t-x)d\mu(t)
\]
does not have an accumulation point at infinity. 
\end{enumerate} 
\end{theorem}
\begin{proof}
The analysis of the operator $\widetilde{J}(x)$ basically goes the same lines as it was done for the operator $\widetilde{J}_0(\cos\pi\alpha)$.
So, let us sketch the proof.

(i)$\Rightarrow$(ii)  If $\widetilde{J}(x)$ is definitizable then $\rho(\widetilde{J}(x))\ne\emptyset$.
Consequently, applying Proposition \ref{HKM_th} and the definition of $\widetilde{J}(x)$ gives \eqref{SpecD}.
 
(ii)$\Rightarrow$(iii) If \eqref{SpecD} holds true then $\widetilde{J}(x)$ is self-adjoint in $\ell^2(G)$. Moreover, $\widetilde{J}(x)$ is definitizable because one can easily mimic \eqref{NonDar}
in the general case. Next, from Propostion \ref{SpecDef} one gets that $\widetilde{J}(x)$ is bounded. Thus, $L$ is also bounded which implies
\eqref{Bcond}. Therefore, the desired statement is a consequence of Proposition \ref{DDth}.

(iii)$\Rightarrow$(i) Proposition \ref{DDth} and \eqref{LU_entries1}, \eqref{LU_entries} give the boundedness of $L$.
So, it follows from general version of \eqref{NonDar} and the boundedness of $\widetilde{J}(x)$ that the operator $\widetilde{J}(x)$ is definitizable.
\end{proof}

\begin{remark}
Taking into account Remark \ref{SimonEx} and Theorem \ref{Berlin}, one can see that the definitizability plays 
the same role for signed measures as the self-adjointness does for positive measures. As we saw, using only the  self-adjointness in the case
of signed measures doesn't help a lot. At the same time, using definitizability allows to prove Markov-type results for 
the Cauchy transforms of signed measures with one sign change  \cite[Theorem 5.5]{DD09}. It would also be nice to extend this theory to the case
of multiple sign changes.
\end{remark}

\noindent{\bf Acknowledgments.}
I gratefully acknowledge the support of FWO Flanders project G.0934.13 and Belgian Interuniversity Attraction Pole P07/18.
Also, this research was partially supported by the European Research Council under the European Union  Seventh
Framework Programme (FP7/2007-2013)/ERC grant agreement \linebreak no. 259173 and was partially done in Berlin,
where Herbert Stahl and Peter Jonas were studying different faces of the same object  at about the same time but on the different
sides of the wall.

\end{document}